\newfont{\bcb}{msbm10}
\newfont{\matb}{cmbx10}
\newfont{\got}{eufm10}

\documentclass[12pt]{amsart}
\usepackage{amsmath, amsthm, amscd, amsfonts, amssymb, latexsym, graphicx, color}
\usepackage[bookmarksnumbered, colorlinks, plainpages, hypertex]{hyperref}

\usepackage[cp1250]{inputenc}

\usepackage{amsmath,amsthm}
\usepackage{amssymb,latexsym}
\usepackage{enumerate}

\newtheorem{theorem}{Theorem}[section]
\newtheorem{lemma}[theorem]{Lemma}
\newtheorem{proposition}[theorem]{Proposition}
\newtheorem{corollary}[theorem]{Corollary}
\theoremstyle{definition}

\theoremstyle{remark}
\newtheorem{remark}[theorem]{Remark}
\numberwithin{equation}{section}

\begin{document}

\title[The closedness theorem]{The closedness theorem \\ over Henselian valued fields}

\author[Krzysztof Jan Nowak]{Krzysztof Jan Nowak}





\begin{abstract}
We prove the closedness theorem over Henselian valued fields,
which was established over rank one valued fields in one of our
recent papers. In the proof, as before, we use the local behaviour
of definable functions of one variable and the so-called fiber
shrinking, which is a relaxed version of curve selection. Now our
approach applies also relative quantifier elimination for ordered
abelian groups due to Cluckers--Halupczok. Afterwards the
closedness theorem will allow us to achieve i.a.\ the
\L{}ojasiewicz inequality, curve selection and extending
hereditarily rational functions as well as to develop the theory
of regulous functions and sheaves.
\end{abstract}

\maketitle

\section{Introduction}

Throughout the paper, $K$ will be an arbitrary Henselian valued
field of equicharacteristic zero with valuation $v$, value group
$\Gamma$, valuation ring $R$ and residue field $\Bbbk$. Examples
of such fields are the quotient fields of the rings of formal
power series and of Puiseux series with coefficients from a field
$\Bbbk$ of characteristic zero as well as the fields of Hahn
series (maximally complete valued fields also called
Malcev--Neumann fields; cf.~\cite{Kap}):
$$ \Bbbk((t^{\Gamma})) := \left\{ f(t) = \sum_{\gamma \in \Gamma} \
   a_{\gamma}t^{\gamma} : \ a_{\gamma} \in \Bbbk, \ \text{supp}\,
   f(t) \ \text{is well ordered} \right\}. $$
We consider the ground field $K$ along with the three-sorted
language $\mathcal{L}$ of Denef--Pas (cf.~\cite{Pa1,Now2}). Every
valued field $K$ has a topology induced by its valuation $v$.
Cartesian products $K^{n}$ are equipped with the product topology
and subsets of Cartesian products $K^{n}$ inherit a topology,
called the $K$-topology. The main purpose of this paper is to
prove the following closedness theorem.

\begin{theorem}\label{clo-th}
Let $D$ be an $\mathcal{L}$-definable subset of $K^{n}$. Then the
canonical projection
$$ \pi: D \times R^{m} \longrightarrow D  $$
is definably closed in the $K$-topology, i.e.\ if $B \subset D
\times R^{m}$ is an $\mathcal{L}$-definable closed subset, so is
its image $\pi(B) \subset D$.
\end{theorem}

In the case where the ground field $K$ is of rank one, it was
established in our paper~\cite{Now2}. Of course, when $K$ is a
locally compact field, the closedness theorem holds by a routine
topological argument. In the proof given in Section~4, we use, as
before, the local behaviour of definable functions of one variable
and the so-called fiber shrinking, which is a relaxed version of
curve selection. The former result over arbitrary Henselian valued
fields was achived in the paper~\cite[Proposition~5.1]{Now5}. The
proof of the latter, in turn, is given in Section~2. Now our
approach applies also relative quantifier elimination for ordered
abelian groups due to Cluckers--Halupczok~\cite{C-H}, which is
recalled in Section~3.

\begin{remark}
Not all valued fields $K$ have an angular component map, but it
exists if $K$ has a cross section, which happens whenever $K$ is
$\aleph_{1}$-saturated (cf.~\cite[Chap.~II]{Ch}). Moreover, a
valued field $K$ has an angular component map whenever its residue
field $\Bbbk$ is $\aleph_{1}$-saturated
(cf.~\cite[Corollary~1.6]{Pa2}). In general, unlike for $p$-adic
fields and their finite extensions, adding an angular component
map does strengthen the family of definable sets. The $K$-topology
is, of course, definable in the language of valued fields, and
therefore the closedness theorem is a first order property. Hence
it is valid over arbitrary Henselian valued fields of
equicharacteristic zero, because it can be proven using saturated
elementary extensions and one may thus assume that an angular
component map exists.
\end{remark}

Afterwards Theorem~\ref{clo-th} will allow us to establish i.a.\
the \L{}ojasiewicz inequality, curve selection and extending
hereditarily rational functions as well as to develop the theory
of regulous functions and sheaves. Those results were achieved
over rank one valued fields in our paper~\cite{Now2}. The theory
of hereditarily rational functions on the real and $p$-adic
varieties was provided in the paper~\cite{K-N}. The closedness
theorem immediately yields five corollaries stated below. One of
them, the descent property (Corollary~\ref{clo-th-cor-4}), enables
application of resolution of singularities and transformation to a
normal crossing by blowing up in much the same way as over the
locally compact ground field. Other its applications are provided
in our recent papers~\cite{Now3,Now4}.

\begin{corollary}\label{clo-th-cor-1}
Let $D$ be an $\mathcal{L}$-definable subset of $K^{n}$ and
$\,\mathbb{P}^{m}(K)$ stand for the projective space of dimension
$m$ over $K$. Then the canonical projection
$$ \pi: D \times \mathbb{P}^{m}(K) \longrightarrow D $$
is definably closed. \hspace*{\fill} $\Box$
\end{corollary}

\begin{corollary}\label{clo-th-cor-0}
Let $A$ be a closed $\mathcal{L}$-definable subset of
$\,\mathbb{P}^{m}(K)$ or $R^{m}$. Then every continuous
$\mathcal{L}$-definable map $f: A \to K^{n}$ is definably closed
in the $K$-topology.
\end{corollary}


\begin{corollary}\label{clo-th-cor-2}
Let $\phi_{i}$, $i=0,\ldots,m$, be regular functions on $K^{n}$,
$D$ be an $\mathcal{L}$-definable subset of $K^{n}$ and $\sigma: Y
\longrightarrow K\mathbb{A}^{n}$ the blow-up of the affine space
$K\mathbb{A}^{n}$ with respect to the ideal
$(\phi_{0},\ldots,\phi_{m})$. Then the restriction
$$ \sigma: Y(K) \cap \sigma^{-1}(D) \longrightarrow D $$
is a definably closed quotient map.
\end{corollary}


\begin{proof} Indeed,  $Y(K)$ can be regarded as a closed algebraic subvariety of
$K^{n} \times \mathbb{P}^{m}(K)$ and $\sigma$ as the canonical
projection.
\end{proof}

\begin{corollary}\label{clo-th-cor-3}
Let $X$ be a smooth $K$-variety, $\phi_{i}$, $i=0,\ldots,m$,
regular functions on $X$, $D$ be an $\mathcal{L}$-definable subset
of $X(K)$ and $\sigma: Y \longrightarrow X$ the blow-up of the
ideal $(\phi_{0},\ldots,\phi_{m})$. Then the restriction
$$ \sigma: Y(K) \cap \sigma^{-1}(D) \longrightarrow D $$
is a definably closed quotient map.  \hspace*{\fill} $\Box$
\end{corollary}


\begin{corollary}\label{clo-th-cor-4} (Descent property)
Under the assumptions of the above corollary, every continuous
$\mathcal{L}$-definable function
$$ g: Y(K) \cap \sigma^{-1}(D) \longrightarrow K $$
that is constant on the fibers of the blow-up $\sigma$ descends to
a (unique) continuous $\mathcal{L}$-definable function $f: D
\longrightarrow K$. \hspace*{\fill} $\Box$
\end{corollary}

\section{Fiber shrinking}

Consider a Henselian valued field $K$ of equicharacteristic zero
along with the three-sorted language $\mathcal{L}$ of Denef--Pas.
In this section, we remind the reader the concept of fiber
shrinking introduced in our paper~\cite[Section~6]{Now2}.

\vspace{1ex}

Let $A$ be an $\mathcal{L}$-definable subset of $K^{n}$ with
accumulation point
$$ a = (a_{1},\ldots,a_{n}) \in K^{n} $$
and $E$ an $\mathcal{L}$-definable subset of $K$ with accumulation
point $a_{1}$. We call an $\mathcal{L}$-definable family of sets
$$ \Phi = \bigcup_{t \in E} \ \{ t \} \times \Phi_{t} \subset A $$
an $\mathcal{L}$-definable $x_{1}$-fiber shrinking for the set $A$
at $a$ if
$$ \lim_{t \rightarrow a_{1}} \, \Phi_{t} = (a_{2},\ldots,a_{n}),
$$
i.e.\ for any neighbourhood $U$ of $(a_{2},\ldots,a_{n}) \in
K^{n-1}$, there is a neighbourhood $V$ of $a_{1} \in K$ such that
$\emptyset \neq \Phi_{t} \subset U$ for every $t \in V \cap E$, $t
\neq a_{1}$. When $n=1$, $A$ is itself a fiber shrinking for the
subset $A$ of $K$ at an accumulation point $a \in K$.


\begin{proposition}\label{FS} (Fiber shrinking)
Every $\mathcal{L}$-definable subset $A$ of $K^{n}$ with
accumulation point $a \in K^{n}$ has, after a permutation of the
coordinates, an $\mathcal{L}$-definable $x_{1}$-fiber shrinking at
$a$.
\end{proposition}

In the case where the ground field $K$ is of rank one, the proof
of Proposition~\ref{FS} was given in~\cite[Section~6]{Now2}. In
the general case, it can be repeated verbatim once we demonstrate
the following result on definable subsets in the value group sort
$\Gamma$.

\begin{lemma}\label{line}
Let $\Gamma$ be an ordered abelian group and $P$ be a definable
subset of $\Gamma^{n}$. Suppose that $(\infty,\ldots,\infty)$ is
an accumulation point of $P$, i.e.\ for any $\delta \in \Gamma$
the set
$$ \{ x \in P: x_{1} > \delta, \ldots, x_{n} > \delta \} \neq \emptyset $$
is non-empty. Then there is an affine semiline $L$ passing through
a point $\gamma = (\gamma_{1},\ldots,\gamma_{n}) \in P$:
$$ L = \{ (r_{1}t + \gamma_{1},\ldots,r_{n}t + \gamma_{n}): \, t
   \in \Gamma, \ t \geq 0 \}, $$
where $r_{1},\ldots,r_{n} \in \mathbb{N}$ are positive integers,
such that $(\infty,\ldots,\infty)$ is an accumulation point of the
intersection $P \cap L$ too.
\end{lemma}

In~\cite[Section~6]{Now2}, Lemma~\ref{line} was shown for
archimedean groups by means of quantifier elimination in the
Presburger language. But in the general case, it follows in a
similar fashion via relative quantifier elimination for ordered
abelian groups (apply Theorem~\ref{RQE} along with
Remarks~\ref{Rem1} and~\ref{Rem2}), recalled in the next section.

\section{Quantifier elimination for ordered abelian
groups}

It is well known that archimedean ordered abelian groups admit
quantifier elimination in the Presburger language. Much more
complicated are quantifier elimination results for non-archimedean
groups (especially those with infinite rank), going back as far as
Gurevich~\cite{Gur}. He established a transfer of sentences from
ordered abelian groups to so-called coloured chains (i.e.\
linearly ordered sets with additional unary predicates), enhanced
later to allow arbitrary formulas. This was done in his doctoral
dissertation "The decision problem for some algebraic theories"
(Sverdlovsk, 1968), and next also by Schmitt in his habilitation
thesis "Model theory of ordered abelian groups" (Heidelberg,
1982); see also the paper~\cite{Sch}. Such a transfer is a kind of
relative quantifier elimination, which allows
Gurevich--Schmitt~\cite{G-S} in their study of the NIP property to
lift model theoretic properties from ordered sets to ordered
abelian groups or, in other words, to transform statements on
ordered abelian groups into those on coloured chains.

\vspace{1ex}

Instead Cluckers--Halupczok~\cite{C-H} introduce a suitable
many-sorted language $\mathcal{L}_{qe}$ with main group sort
$\Gamma$ and auxiliary imaginary sorts which carry the structure
of a linearly ordered set with some additional unary predicates.
They provide quantifier elimination relative to the auxiliary
sorts, where each definable set in the group is a union of a
family of quantifier free definable sets with parameter running a
definable (with quantifiers) set of the auxiliary sorts.

\vspace{1ex}

Fortunately, sometimes it is possible to directly deduce
information about ordered abelian groups without any knowledge of
the auxiliary sorts. For instance, this may be illustrated by
their theorem on piecewise linearity of definable
functions~\cite[Corollary~1.10]{C-H} as well as by
Proposition~\ref{line} and application of quantifier elimination
in the proof of the closedness theorem in Section~4.

\vspace{1ex}

Now we briefly recall the language $\mathcal{L}_{qe}$ taking care
of points essential for our applications.

\vspace{1ex}

The main group sort $\Gamma$ is with the constant $0$, the binary
function $+$ and the unary function $-$. The collection
$\mathcal{A}$ of auxiliary sorts consists of certain imaginary
sorts:
$$ \mathcal{A} := \{ \mathcal{S}_{p}, \mathcal{T}_{p},
   \mathcal{T}^{+}_{p}: p \in \mathbb{P} \}; $$
here $\mathbb{P}$ stands for the set of prime numbers. By abuse of
notation, $\mathcal{A}$ will also denote the union of the
auxiliary sorts. In this section, we denote $\Gamma$-sort
variables by $x,y,z,\ldots$ and auxiliary sorts variables by
$\eta, \theta, \zeta, \ldots$.

\vspace{1ex}

Further, the language $\mathcal{L}_{qe}$ consists of some unary
predicates on $\mathcal{S}_{p}$, $p \in \mathbb{P}$, some binary
order relations on $\mathcal{A}$, a ternary relation
$$ x \equiv_{m,\alpha}^{m'} y \ \ \text{on} \ \
   \Gamma \times \Gamma \times \mathcal{S}_{p} \ \
   \text{for each} \ \ p \in \mathbb{P}, \ m,m' \in \mathbb{N}, $$
and finally predicates for the ternary relations
$$ x \diamond_{\alpha} y + k_{\alpha} \ \ \text{on} \ \ \Gamma \times \Gamma
   \times \mathcal{A}, $$
where $\diamond \in \{ =, <, \equiv_{m} \}$, $m \in \mathbb{N}$,
$k \in \mathbb{Z}$ and $\alpha$ is the third operand running any
of the auxiliary sorts $\mathcal{A}$.

\vspace{1ex}

We now explain the meaning of the above ternary relations, which
are defined by means of certain definable subgroup
$\Gamma_{\alpha}$ and $\Gamma_{\alpha}^{m'}$ of $\Gamma$ with
$\alpha \in \mathcal{A}$ and $m' \in \mathbb{N}$. Namely we write
$$ x \equiv_{m,\alpha}^{m'} y \ \ \text{iff} \ \  x-y \in
   \Gamma_{\alpha}^{m'} + m\Gamma. $$
Further, let $1_{\alpha}$ denote the minimal positive element of
$\Gamma/\Gamma_{\alpha}$ if $\Gamma/\Gamma_{\alpha}$ is discrete
and $1_{\alpha} :=0$ otherwise, and set $k_{\alpha} := k \cdot
1_{\alpha}$ for all $k \in \mathbb{Z}$. By definition we write
$$ x \diamond_{\alpha} y + k_{\alpha} \ \ \ \text{iff} \ \ \
   x \, (\bmod \, \Gamma_{\alpha}) \diamond y \ (\bmod \, \Gamma_{\alpha}) +
   k_{\alpha}. $$
(Thus the language $\mathcal{L}_{qe}$ incorporates the Presburger
language on all quotients $\Gamma/\Gamma_{\alpha}$.) Note also
that the ordinary predicates $<$ and $\equiv_{m}$ on $\Gamma$ are
$\Gamma$-quantifier-free definable in the language
$\mathcal{L}_{qe}$.

\vspace{1ex}

Now we can readily formulate quantifier elimination relative to
the auxiliary sorts (\cite[Theorem~1.8]{C-H}).

\begin{theorem}\label{RQE}
In the theory $T$ of ordered abelian groups, each
$\mathcal{L}_{qe}$-formula $\phi(\bar{x},\bar{\eta})$ is
equivalent to an $\mathcal{L}_{qe}$-formula
$\psi(\bar{x},\bar{\eta})$ in family union form, i.e.\
$$ \psi(\bar{x},\bar{\eta}) = \bigvee_{i=1}^{k} \; \exists \,
   \bar{\theta} \: \left[ \chi_{i}(\bar{\eta},\bar{\theta})
   \wedge \omega_{i}(\bar{x},\bar{\theta}) \right], $$
where $\bar{\theta}$ are $\mathcal{A}$-variables, the formulas
$\chi_{i}(\bar{\eta},\bar{\theta})$ live purely in the auxiliary
sorts $\mathcal{A}$, each $\omega_{i}(\bar{x},\bar{\theta})$ is a
conjunction of literals (i.e.\ atomic or negated atomic formulas)
and $T$ implies that the $\mathcal{L}_{qe}(\mathcal{A})$-formulas
$$ \{ \chi_{i}(\bar{\eta},\bar{\alpha}) \wedge
   \omega_{i}(\bar{x},\bar{\alpha}): \; i=1,\ldots,k, \
   \bar{\alpha} \in \mathcal{A} \} $$
are pairwise inconsistent.
\end{theorem}

\begin{remark}\label{Rem1}
The sets definable (or, definable with parameters) in the main
group sort $\Gamma$ resemble to some extent the sets which are
definable in the Presburger language. Indeed, the atomic formulas
involved in the formulas $\omega_{i}(\bar{x},\bar{\theta})$ are of
the form
$$ t(\bar{x}) \diamond_{\theta_{j}} k_{\theta_{j}}, $$
where $t(\bar{x})$ is a $\mathbb{Z}$-linear combination
(respectively, a $\mathbb{Z}$-linear combination plus an element
of $\Gamma$) , the predicates
$$ \diamond \in \{ =, <, \equiv_{m}, \equiv_{m}^{m'} \} \ \ \text{with some}
   \ \ m,m' \in \mathbb{N}, $$
$\theta_{j}$ is one of the entries of $\bar{\theta}$ and $k \in
\mathbb{Z}$; here $k=0$ if $\diamond$ is $\equiv_{m}^{m'}$.
Clearly, equality and inequalities define polyhedra and congruence
conditions define sets which consist of entire cosets of $m\Gamma$
for finitely many $m \in \mathbb{N}$.
\end{remark}

\begin{remark}\label{Rem2}
Note also that the sets given by atomic formulas $t(\bar{x})
\diamond_{\theta_{j}} k_{\theta_{j}}$ consist of entire cosets of
the subgroups $\Gamma_{\theta_{j}}$. Therefore, the union of those
subgroups $\Gamma_{\theta_{j}}$ which essentially occur in a
formula in family union form, describing a proper subset of
$\Gamma^{n}$, is not cofinal with $\Gamma$. This observation is
often useful as, for instance, in the proofs of fiber shrinking
and Theorem~\ref{clo-th}.
\end{remark}

\section{Proof of the closedness theorem}

Generally, we shall follow the idea of the proof
from~\cite[Section~7]{Now2}. Again, the proof reduces easily to
the case $m=1$ and next, by means of fiber shrinking
(Proposition~\ref{FS}), to the case $n=1$ and $a=0 \in K$.

\vspace{1ex}

Whereas in the paper~\cite{Now2} preparation cell decomposition
(due to Pas; see \cite[Theorem~3.2]{Pa1}
and~\cite[Theorem~2.4]{Now2}) was combined with quantifier
elimination in the $\Gamma$ sort in the Presburger language, here
it is combined with relative quantifier elimination due to
Cluckers--Halupczok. In the same manner as before, we can now
assume that $B$ is a subset $F$ of a cell $C$, as explained below.
Let
$$ a(x,\xi),b(x,\xi),c(x,\xi): \, D \longrightarrow K $$
be three $\mathcal{L}$-definable functions on an
$\mathcal{L}$-definable subset $D$ of $K^{2} \times \Bbbk^{m}$ and
let $\nu \in \mathbb{N}$ is a positive integer. For each $\xi \in
\Bbbk^{m}$ set
$$ C(\xi) := \left\{ \rule{0em}{3ex} (x,y) \in K^{n}_{x} \times K_{y}:
   \ (x,\xi) \in D, \right. $$
$$ \left. \rule{0em}{3ex} v(a(x,\xi)) \lhd_{1} v((y-c(x,\xi))^{\nu}) \lhd_{2} v(b(x,\xi)), \
   \overline{ac} (y-c(x,\xi)) = \xi_{1} \right\}, $$
where $\lhd_{1},\lhd_{2}$ stand for $<, \leq$ or no condition in
any occurrence. A cell $C$ is by definition a disjoint union of
the fibres $C(\xi)$. The subset $F$ of $C$ is a union of fibers
$F(\xi)$ of the form
$$ F(\xi) := \left\{ \rule{0em}{4ex} (x,y) \in C(\xi): \ \exists \
   \bar{\theta} \ \chi(\bar{\theta}) \ \wedge\ \right. $$
$$ \bigwedge_{i \in I_{a}} \
   v(a_{i}(x,\xi)) \lhd_{1,\theta_{j_{i}}} v((y - c(x,\xi))^{\nu_{i}}), \
   \bigwedge_{i \in I_{b}} \
   v((y - c(x,\xi))^{\nu_{i}}) \lhd_{2,\theta_{j_{i}}} v(b_{i}(x,\xi)) $$
$$  \left. \wedge \ \bigwedge_{i \in I_{f}} \
    v((y - c(x,\xi))^{\nu_{i}}) \diamond_{\theta_{j_{i}}} v(f_{i}(x,\xi)) \right\}, $$
where $I_{a}$, $I_{b}$, $I_{f}$ are finite (possibly empty) sets
of indices, $a_{i}$, $b_{i}$, $f_{i}$ are $\mathcal{L}$-definable
functions, $\nu_{i},M \in \mathbb{N}$ are positive integers,
$\lhd_{1}$, $\lhd_{2}$ stand for $<$ or $\leq$, the predicates
$$ \diamond \in \{ \equiv_{M}, \neg \equiv_{M}, \equiv_{M}^{m'}, \neg \equiv_{M}^{m'} \}
   \ \ \text{with some} \ \ m' \in \mathbb{N}, $$
and $\theta_{j_{i}}$ is one of the entries of $\bar{\theta}$.

\vspace{1ex}

As before, since every $\mathcal{L}$-definable subset in the
Cartesian product $\Gamma^{n} \times \Bbbk^{m}$ of auxiliary sorts
is a finite union of the Cartesian products of definable subsets
in $\Gamma^{n}$ and in $\Bbbk^{m}$, we can assume that $B$ is one
fiber $F(\xi')$ for a parameter $\xi' \in \Bbbk^{m}$. For
simplicity, we abbreviate
$$ c(x,\xi'), a(x,\xi'), b(x,\xi'), a_{i}(x,\xi'), b_{i}(x,\xi'), f_{i}(x,\xi') $$
to
$$ c(x), a(x), b(x), a_{i}(x), b_{i}(x), f_{i}(x) $$
with $i \in I_{a}$, $i \in I_{b}$ and $i \in I_{f}$. Denote by $E
\subset K$ the common domain of these functions; then $0$ is an
accumulation point of $E$.

\vspace{1ex}

By the theorem on existence of the limit, established over
arbitrary Henselian valued fields in our
paper~\cite[Proposition~5.1]{Now5}, we can assume that the limits
$$ c(0), a(0), b(0), a_{i}(0), b_{i}(0), f_{i}(0) $$
of the functions
$$ c(x), a(x), b(x), a_{i}(x), b_{i}(x), f_{i}(x) $$
when $x \rightarrow 0\,$ exist in $R$. Moreover, there is a
neighbourhood $U$ of $0$ such that, each definable set
$$ \{ (v(x), v(f_{i}(x))): \; x \in (E \cap U) \setminus \{0 \} \}
   \subset \Gamma \times (\Gamma \cup \ \{
   \infty \}), \ \ i \in I_{f},  $$
is contained in an affine line with rational slope
\begin{equation}\label{affine}
l = \frac{p_{i}}{q} \cdot k + \beta_{i}, \ \ i \in I_{f},
\end{equation}
with $p_{i},q \in \mathbb{Z}$, $q>0$, $\beta_{i} \in \Gamma$, or
in\/ $\Gamma \times \{ \infty \}$.

\vspace{1ex}

The role of the center $c(x)$ is, of course, immaterial. We may
assume, without loss of generality, that it vanishes, $c(x) \equiv
0$, for if a point $b = (0,w) \in K^{2}$ lies in the closure of
the cell with zero center, the point $(0, w + c(0))$ lies in the
closure of the cell with center $c(x)$.

\vspace{1ex}

Observe now that If $\lhd_{1}$ occurs and $a(0)= 0$, the set
$F(\xi')$ is itself an $x$-fiber shrinking at $(0,0)$ and the
point $b=(0,0)$ is an accumulation point of $B$ lying over $a=0$,
as desired. And so is the point $b=(0,0)$ if
$\lhd_{1,\theta_{j_{i}}}$ occurs and $a_{i}(0)= 0$ for some $i \in
I_{a}$, because then the set $F(\xi')$ contains the $x$-fiber
shrinking
$$ F(\xi') \cap \{ (x,y) \in E \times K: \
   v(a_{i}(x)) \lhd_{1} v(y^{\nu_{i}}) \}. $$

\vspace{1ex}

So suppose that either only $\lhd_{2}$ occur or $\lhd_{1}$ occur
and, moreover, $a(0) \neq 0$ and $a_{i}(0) \neq 0$ for all $i \in
I_{a}$. By elimination of $K$-quantifiers, the set $v(E)$ is a
definable subset of $\Gamma$. Further, it is easy to check,
applying Theorem~\ref{RQE} ff.\ likewise as it was in
Lemma~\ref{line}, that the set $v(E)$ is given near infinity only
by finitely many congruence conditions of the form
\begin{equation}\label{vE}
  v(E) = \left\{ k \in \Gamma: \ k > \beta \ \wedge \ \exists \
  \bar{\theta} \ \, \omega(\bar{\theta}) \ \wedge \
  \bigwedge_{i=1}^{s} \ m_{i} k \, \diamond_{N,\theta_{j_{i}}} \gamma_{i}
  \right\}.
\end{equation}
where $\beta, \gamma_{i} \in \Gamma$, $m_{i},N \in \mathbb{N}$ for
$i=1,\ldots,s$, the predicates
$$ \diamond \in \{ \equiv_{N}, \neg \equiv_{N}, \equiv_{N}^{m'}, \neg \equiv_{N}^{m'}
   \} \ \ \text{with some} \ \ m' \in \mathbb{N}, $$
and $\theta_{j_{i}}$ is one of the entries of $\bar{\theta}$.
Obviously, after perhaps shrinking the neighbourhood of zero, we
may assume that
$$ v(a(x)) = v(a(0)) \ \ \text{and} \ \ v(a_{i}(x)) = v(a_{i}(0)) $$
for all $i \in I_{a}$ and $x \in E \setminus \{ 0 \}$, $v(x)
> \beta$.

\vspace{1ex}

Now, take an element $(u,w) \in F(\xi')$ with $u \in E \setminus
\{ 0 \}$, $v(u) > \beta$. In order to complete the proof, it
suffices to show that $(0,w)$ is an accumulation point of
$F(\xi')$. To this end, observe that, by equality~\ref{vE}, there
is a point $x \in E$ arbitrarily close to $0$ such that
$$ v(x) \in v(u) + q M N \cdot \Gamma. $$
By equality~\ref{affine}, we get
$$ v(f_{i}(x)) \in v(f_{i}(u)) + p_{i} M N \cdot \Gamma, \ \ \ i \in I_{f}, $$
and hence
\begin{equation}\label{vf}
   v\left( f_{i}(x) \right) \equiv_{M} v(f_{i}(u)), \ \ \ i \in I_{f}.
\end{equation}
Clearly, in the vicinity of zero we have
$$ v(y^{\nu}) \lhd_{2} v(b(x,\xi)) $$
and
$$ \bigwedge_{i \in I_{b}} \
   v(y^{\nu_{i}}) \lhd_{2,\theta_{j_{i}}} v(b_{i}(x,\xi)). $$
Therefore equality~\ref{vf} along with the definition of the fibre
$F(\xi')$ yield $(x,w) \in F(\xi')$, concluding the proof.
 \hspace*{\fill} $\Box$

\vspace{2ex}

\begin{small}
Institute of Mathematics

Faculty of Mathematics and Computer Science

Jagiellonian University


ul.~Profesora \L{}ojasiewicza 6, 30-348 Krak\'{o}w, Poland

{\em E-mail address: nowak@im.uj.edu.pl}
\end{small}

\end{document}